\documentclass[10pt,twoside,a4paper]{article} 
\usepackage{amsfonts}
\usepackage[utf8]{inputenc}
\usepackage[T1]{fontenc}
\usepackage{lmodern}

\usepackage[bookmarksopen=false,breaklinks=true,%
      backref=page,pagebackref=true,plainpages=false,%
      hyperindex=true,pdfstartview=FitH,%
      pdfpagelabels=true,colorlinks=true,linkcolor=blue,%
      citecolor=red,urlcolor=red,hypertexnames=false,colorlinks=false%
      ]%
   {hyperref}

\newtheorem{theorem}{Theorem}[section]

\newtheorem{proposition}[theorem]{Proposition}
\newtheorem{lemma}[theorem]{Lemma}
\newtheorem{corollary}[theorem]{Corollary}

\newtheorem{definition}[theorem]{Definition}

\newtheorem{notation}[theorem]{Notation}

\newcommand {\junk}[1]{}

\newenvironment{proof}[1]{
\trivlist \item[\hskip \labelsep{\bf #1.}]\hskip 0pt\\}{\hfill
\mbox{$\Box$}
\endtrivlist}
\makeatletter

\def\.@{\char'76}

\def \noi {\noindent}
\def \ss {\smallskip}
\def \sni {\ss\noi}
\def \ms {\medskip}
\def \mni {\ms\noi}
\def \bs {\bigskip}
\def \bni {\bs\noi}

\def \aoe {{\land,\lor,\exists}}

\def \these {\quad\vdash\quad}
\def \thesu {\quad\vdash_\aoe\quad}

\def \cl {{\circ}}

\def\equidef{\buildrel{{\rm def}}\over{\quad\Longleftrightarrow\quad}}
\def\Equ{\quad\Longleftrightarrow\quad}

\def\gen#1{\left\langle{#1}\right\rangle}

\def \snic#1 {\sni\centerline{$#1$}\ss}
\def \snif#1#2#3 {\vspace{#1}\noindent\centerline{$#3$}\vspace{#2}}

\def \fait#1#2 {\vspace{.1cm} \begin{tabular}{p{2cm}p{2cm}l l l}
$\cl\;#1\;\cl$ & $\these$ & $#2$ \end{tabular}\vspace{.1cm}}
\def \faut#1#2 {\vspace{.1cm} \begin{tabular}{p{2cm}p{2cm}l l l}
                $\cl\;#1\;\cl$ & $\thesu$ & $#2$
                   \end{tabular}}

\def \ov#1 {\overline#1 }

\def \wi#1 {\widetilde#1 }
\def \sd#1#2 {\widetilde#1_#2 }

\def \N{\mathbb{N}}

\def \cC {{\cal C}}

\def \cI {{\cal I}}

\def \cP {{\cal P}}

\def \cM {{\cal M}}

\def \vu {\,\vee\,} 
\def \vi {\,\wedge\,} 
\def \Vu {\bigvee}
\def \Vi {\bigwedge}
\def \vda {\,\vdash\,}

\def \Un {{\bf 1}}
\def \Deux {{\bf 2}}
\def \Trois {{\bf 3}}
\def \Quatre {{\bf 4}}

\def \Pf {{{\rm P}_{{\rm f}}}}

\def \Hom {{\rm Hom}}

\def \dim {{\rm dim}}

\def \Spec {{\rm Spec}}
\def \Zar {\,{\rm Zar}}
\def \Kr {\,{\rm Kr}}
\def \Kru {\,{\rm Kru}}

\def \cad {\textit{i.e.}, }
\def \ssi {if, and only if, }

\def \Propeq {The following properties are equivalent: }
\def \propeq {the following properties are equivalent: }
\def \disept {17$^{{\rm th}}$ Hilbert's problem }

\def \entrel {entailment relation }
\def \entrelz {entailment relation}
\def \entrels {entailment relations }
\def \entrelsz  {entailment relations}

\def \homo {homomorphism }

\def \homos {homomorphisms }

\def \nst {Null\-stellen\-satz }
\def \nstz {Null\-stellen\-satz}

\def \proi {idealistic prime }
\def \prois {idealistic primes }
\def \proiz {idealistic prime}
\def \proisz {idealistic primes}

\def \proc {idealistic chain }

\def \procz {idealistic chain}

\def \proel {elementary \proc}

\def \proelz {elementary \procz}

\def \prolo {\proc of length }

\def \trdi {distributive lattice }

\def \trdisz  {distributive lattices}

\def \LLPO{{\bf LLPO}}

\def \tcg {completeness theorem }
\def \tcgz {completeness theorem}

\def \Tcgi {The \tcg implies the following result. }

\RequirePackage{etoolbox}

\begin{document}
\title{ Hidden constructions in abstract algebra \\
Krull Dimension of distributive lattices and commutative rings}
\author{
Thierry Coquand
(\thanks {~
Chalmers, University of G\"oteborg, Sweden,
email: coquand@cse.gu.se}~)
Henri Lombardi
(\thanks{~
\'Equipe de Math\'ematiques, CNRS UMR 6623, UFR des Sciences et
Techniques,
Universit\'e de Franche-Comt\'e, 25030 Besan\c con cedex, FRANCE,
email: lombardi@univ-fcomte.fr}~),
}
\date{may 2002}

\maketitle

\begin{abstract}
We present constructive versions of Krull's dimension theory for
commutative rings and distributive lattices. The foundations of these
constructive versions are due to Joyal, Espa\~nol and the authors.
We show that the notion of Krull dimension
has an explicit computational content in the form of existence (or
lack of existence) of some algebraic identities.
We can then get an explicit computational content where abstract
results about dimensions are
used to show the existence of concrete elements. This can be seen
as a partial realisation of Hilbert's program for classical abstract
commutative algebra.
\end{abstract}
\bni MSC 2000: 13C15, 03F65, 13A15, 13E05

\bni Key words: Krull dimension, distributive lattices,
Constructive Mathematics.

\bigskip This paper has been published as: {\em Hidden constructions in abstract algebra: {K}rull dimension of
distributive lattices and commutative rings}.  p.\ 477--499 in Commutative ring theory and applications (Fez, 2001), Lecture Notes in Pure and Applied Mathematics, volume 231. Dekker, New-York (2003).

Here we have updated the bibliography.

Notice also that for entailment relations and from a constructive viewpoint it is simpler to take~$\Pf(S)$
as the set of finitely enumerated subsets of $S$ (rather than finite subsets).

\newpage

\tableofcontents

\newpage
\markboth{Introduction}{Introduction}

\section*{Introduction} \label{sec Introduction}
\addcontentsline{toc}{section}{Introduction}

We present constructive versions of Krull's dimension theory for
commutative rings and distributive lattices. The foundations of these
constructive versions are due to Joyal, Espa\~nol and the authors.
We show that the notion of Krull dimension
has an explicit computational content in the form of existence (or
lack of existence) of some algebraic identities. This confirms the feeling
that
commutative algebra can be seen computationally as a machine that produces
algebraic identities (the most famous of which being called \nstz).
This can be seen
as a partial realisation of Hilbert's program for classical abstract
commutative algebra.

  Our presentation follows Bishop's style (cf. in algebra \cite{MRR}).
As much as possible, we kept minimum any explicit mention to logical
notions.
When we say that we have a constructive version of an abstract
algebraic theorem, this means that we have a theorem the proof of
which is constructive, which has a clear computational content, and
from which we can recover the usual version of the abstract theorem
by an immediate application of a well classified non-constructive
principle. An abstract
classical theorem can have several distinct interesting constructive
versions.

  In the case of abstract theorem in commutative algebra, such a
non-constructive principle is the completeness theorem, which
claims the existence of a model of a formally consistent propositional
theory.
We recall the exact formulation of this theorem in the appendix, as
well as its derivation from the compactness theorem
When this is used for algebraic structures of enumerable
presentation (in a suitable sense) the compactness and
completeness theorem can be seen as
a reformulation of Bishop $\LLPO$ (a real number is $\geq 0$ or $\leq 0$).

  To avoid the use of \tcg is not
motivated by philosophical but by practical
considerations. The use of this principle leads indeed to replace
quite direct (but usually hidden) arguments by indirect ones which are
nothing else than a double contraposition of the direct proofs,
with a corresponding lack of computational content.
For instance \cite{clr} the abstract proof of \disept claims~: if
the polynomial $P$ is not a sum of rational fractions there is
a field $K$ in which one can find an absurdity by reading the
(constructive)
proof that the polynomial is everywhere positive or zero. The direct
version of this abstract proof is: from the (constructive) proof
that  the polynomial is everywhere positive or zero, one can show
(using arguments of the abstract proofs) that any attempt to build $K$
will fail. This gives explicitly the sum of squares we are looking for.
In the meantime, one has to replace the abstract result: ``any real
field can be ordered'' by the constructive theorem: ``in a field
in which any attempt to build an ordering fails $-1$ is a sum of
squares''.
One can go from this explicit version to the abstract one by \tcgz,
while
the proof of the explicit version is hidden in the
algebraic manipulations that appear in the usual classical proof of the
abstract version.

\mni Here is the content of the paper.


\paragraph{Distributive lattices}~

\noindent
In this section, we present basic theorems on distributive
lattices. An important simplification of proofs and computations
is obtained via the systematic use of the notion of entailment
relation, which has its origin in the cut rule  in
Gentzen's sequent calculus, with the fundamental theorem~\ref{thEntRel1}.

\paragraph{Dimension of distributive lattices}~
 In this section,  we develop the theory of
Krull dimension of distributive lattices, explaining briefly
the connection with Espa\~nol's developments of Joyal' s theory.
We show that  the property to have a Krull dimension $\leq {\ell}$
can be formulated as the existence of concrete equalities in
the distributive lattice.

\paragraph{Zariski and Krull lattice}

\noindent In  section \ref{secZariKrull}  we define the
Zariski lattice of a commutative ring (whose elements are
radicals of finitely generated ideals), which is the constructive
counterpart of Zariski spectrum~: the points of Zariski spectrum
are the prime ideals of Zariski lattice, and the constructible subsets
of Zariski spectrum are the elements of the Boolean algebra generated
by the Zariski lattice.
Joyal's idea is to define Krull dimension of a commutative ring as
the dimension of its Zariski lattice. This avoids any mention of
prime ideals. We show the equivalence between this (constructive) point
of view and the (constructive) presentation given in \cite{lom},
showing that the property to have a Krull dimension $\leq {\ell}$
can be formulated as the existence of concrete equalities in the ring.


\paragraph{Conclusion}~

\noindent
  This article confirms the actual realisation of Hilbert's program
for a large part of abstract commutative algebra.
(cf.  \cite{clr,cp,kl,lom95,lom97,lom98,lom,lom99,lom99a,lq99}).
The general idea is to replace ideal abstract structures
by {\em partial specifications} of these structures.
The very short elegant abstract proof which uses these ideal
objects has then a corresponding computational version at the
level of the partial specifications of these objects.
Most of classical results in abstract commutative algebra,
the proof of which seem to require in an essential way excluded
middle and Zorn's lemma, seem to have in this way a corresponding
constructive version. Most importantly, the abstract proof of the
classical theorem always
contains, more or less implicitly, the constructive proof of
the corresponding constructive version.

\noindent
 Finally, we should note that the explicit characterisations of Krull
dimension of distributive lattices, Theorem \ref{dimLatt}, of
spectral spaces, Theorem \ref{dimSpec}, and of rings,
Corollary \ref{KrulldimRing}, are new.

\patchcmd{\sectionmark}{\MakeUppercase}{}{}{}


\section{Distributive lattice, Entailment relations}
\label{secKrullTreil}
{\em Elementary though it has become after successive
presentations and simplifications, the theory of distributive lattices
is the ideal instance of a mathematical theory, where a syntax is specified
together with a complete description of all models, and what is more, a table
of semantic concepts and syntactic concepts is given, together with
a translation algorithm between the two kinds of concepts. Such an
algorithm is a ``completeness theorem''} (G. C. Rota \cite{Rota}).

\subsection{Distributive lattices, filters and spectrum}
\label{subsecTrd1}
As indicated by the quotation above, the structure
of distributive lattices is fundamental
in mathematics, and G.C. Rota has pointed out repeatedly its potential
relevance to commutative algebra and algebraic geometry.
A distributive lattice is an ordered set with finite sups and infs,
a minimum element (written $0$) and a maximum element (written $1$).
The operations sup and inf are supposed to be distributive w.r.t.
the other. We write these operations
$\vu$ and $\vi$. The relation
  $a\leq b$  can then be defined by $a\vu b= b$ or, equivalently,
$a\vi b = a$. The theory of
distributive
lattices is then purely equational. It makes sense then to talk of
distributive lattices defined by generators and relations.

  A quite important rule, the {\em cut rule}, is the following
$$ \left( ((x\vi a)\; \leq\;  b)\quad\&\quad  (a\; \leq\; (x\vu  b))
\right)\; \Longrightarrow \; a \leq\;  b.
$$
In order to prove this, write $ x\vi a\vi b= x\vi a$  and
$a=  a\vi(x\vu b)$ hence
$$ a= (a\vi x)\vu(a\vi b)= (a\vi x\vi b)\vu(a\vi b)= a\vi b.
$$

A totally ordered set is a distributive lattice as soon as it has
a maximum and a minimum element.
We write ${\bf n}$ for the totally ordered set with
  $n$ elements (this is a distributive lattice for $n\neq 0$.)
A product of distributive lattices is a distributive lattice.
Natural numbers with the divisibility relation form a distributive
lattice
(with minimum element $1$ and maximum element $0$).
If $L$  and $L'$ are two distributive lattices, the set $\Hom(L,L')$ of
all
morphisms
(\cad maps preserving sup, inf, $0$ and $1$) from  $L$ to $L'$ has a
natural order
given by
$$ \varphi \leq \psi \equidef  \forall x\in L\;
\;
\varphi(x) \leq \psi(x).
$$
A map between two totally ordered distributive lattices $L$ and $S$
is a morphism \ssi it is nondecreasing and $0_L$ and $1_L$ are mapped
into
$0_S$ and $1_S$.

The following proposition is direct.
\begin{proposition}
\label{propIdeal} Let $L$ be a distributive lattice and $J$ a subset of
$L$. We consider the distributive lattice $L'$ generated by $L$
and the relations
  $x= 0$ for $x\in J$ ($L'$ is a quotient of $L$).  Then
\begin{itemize}
\item  the equivalence class of $0$ is the set of $a$ such that
for some finite subset $J_0$ of $J$:
$$  a\; \leq\; \Vu_{x\in J_0}x\quad {\rm in} \;  L
$$
%
\item  the equivalence class of $1$ is the set of $b$ such that
for some finite subset $J_0$ of $J$:
$$1\;= \;\left( b\;\vu\;\Vu_{x\in J_0}x\right)\quad {\rm in} \;  L
$$
%
\item  More generally $a\leq_{L'}b$ \ssi
for some finite subset $J_0$ of $J$:
$$  a\; \leq\;  \left( b\; \vu\; \Vu_{x\in J_0}x\right)
$$
\end{itemize}
\end{proposition}

In the previous proposition, the equivalence class of $0$ is called
an {\em ideal} of the lattice; it is the ideal generated by $J$.
We write it $\gen{J}_L$. We can easily check that
an ideal $I$ is a subset such that:
$$\begin{array}{rcl}
   & &  0 \in I   \\
x,y\in I& \Longrightarrow   &  x\vu y \in I   \\
x\in I,\; z\in L& \Longrightarrow   &  x\vi z \in I   \\
\end{array}$$
(the last condition can be written $(x\in I,\;y\leq x)\Rightarrow y\in
I$).

Furthermore, for any morphim
  $\varphi :L_1\rightarrow L_2$,
$\varphi^{-1}(0)$ is an ideal of $L_1$.

A  {\em principal ideal} is an ideal generated by one element $a$.
We have $\gen{a}_L= \{x\in L\; ;\; x\leq a \}$. Any finitely generated
ideal is principal.

  The dual notion of ideal is the one of {\em filter}.
A filter
$F$ is the inverse image of $1$
by a morphism. This is a subset such that:
$$\begin{array}{rcl}
   & &  1 \in F   \\
x,y\in F& \Longrightarrow   &  x\vi y \in F   \\
x\in F,\; z\in T& \Longrightarrow   &  x\vu z \in F   \\
\end{array}$$
\begin{notation}
{\rm We write $\Pf(X)$ for the set of all finite subsets of the set $X$.
If $A$ is a finite subset of a distributive lattice $L$ we define
$$ \Vu A:= \Vu_{x\in A}x\qquad {\rm and}\qquad \Vi A:= \Vi_{x\in A}x
$$
We write $A \vdash B$ or $A \vdash_L B$ for the relation defined on the set
  $\Pf(L)$:
$$ A \vda B \; \; \equidef\; \; \Vi A\;\leq \;
\Vu B
$$
}
\end{notation}
Note the relation  $A \vdash B$ is well defined on finite subsets
because of
associativity commutativity and idempotence of the operations
  $\vi$  and $\vu$.
Note also
$\; \emptyset  \vda \{x\}\; \Rightarrow\;  x= 1\; $ and
$ \{y\} \vda \emptyset\; \Rightarrow \; y= 0$.
This relation satisfies the following axioms, where
we write
  $x$ for $\{x\}$ and $A,B$ for $ A\cup B$.
$$\begin{array}{rcrclll}
&    & a  &\vda& a    &\; &(R)     \\
(A \vda B)\; \& \;(A\subseteq A')\; \& \;(B\subseteq B') &
\; \Longrightarrow \;  & A' &\vda& B'   &\; &(M)     \\
(A,x \vda B)\;
\&
\;(A \vda B,x)  &   \Longrightarrow  & A &\vda& B &\;
&(T)
\end{array}$$
we say that the relation is reflexive,
  \label{remotr} monotone and
transitive. The last rule is also called \emph{cut rule}.
Let us also mention the two following rules of ``distributivity'':
$$\begin{array}{rcl}
(A,\;x \vda B)\;\& \;(A,\;y \vda B)  &  \;  \Longleftrightarrow  \; &
A,\;x\vu y \vda B  \\
(A\vda B,\;x )\;\&\;(A \vda B,\;y)  &   \Longleftrightarrow  &
A\vda B,\;x\vi y
\end{array}$$

The following is proved in the same way as Proposition \ref{propIdeal}.
\begin{proposition}
\label{propIdealFiltre} Let $L$ be a distributive lattice and
$(J,U)$ a pair of subsets of $L$.
We consider the distributive lattice $L'$ generated by $L$ and by the
relations
$x= 0$ for $x\in J$ and $y= 1$ for $y\in U$
($L'$ is a quotient of $L$). We have that:
\begin{itemize}
\item  the equivalence class of $0$ is the set of elements $a$ such
that:
$$ \exists J_0\in\Pf(J),\; U_0\in\Pf(U) \qquad
a,\; U_0 \; \vdash_L\;  J_0
$$
\item  the equivalence class of $1$ is the set of elements $b$ such
that:
v\'erifient:
$$  \exists J_0\in\Pf(J),\; U_0\in\Pf(U)\qquad
  U_0 \; \vdash_L\; b,\; J_0
$$
\item  More generally $a\leq_{L'}b$ \ssi
there exists a finite subset $J_0$ of $J$ and a finite subset $U_0$ of
$U$ such that, in $L$:
$$  a,\; U_0 \; \vdash_L\; b,\; J_0
$$
\end{itemize}
\end{proposition}

We shall write $L/(J= 0,U= 1)$ for the quotient lattice $L'$ described in
Proposition \ref{propIdealFiltre}. Let $\psi:L\rightarrow L'$ be the
canonical surjection. If  $I$ is the ideal $\psi^{-1}(0)$ and $F$ the
filter $\psi^{-1}(1)$, we say that the {\em ideal $I$ and the filter $F$
are conjugate}. By the previous proposition, an ideal $I$ and
a filter $F$ are conjugate \ssi we have:
$$\begin{array}{cl}
  \left[ I_0\in\Pf(I),\, F_0\in\Pf(F), \;
   (x,\; F_0 \vda I_0)\right]\;\Longrightarrow\;  x\in I& \quad{\rm
   and}\
\\
\left[I_0\in\Pf(I),\, F_0\in\Pf(F), \;(F_0 \vda x,\;  I_0)
\right]\;\Longrightarrow\;  x\in F.
\end{array}$$
This can also be formulated as follows:
$$
(f\in F,\; x\vi f \in I) \Longrightarrow x\in I
\quad {\rm and}\quad
(j\in I,\; x\vu j \in F) \Longrightarrow x\in F.
$$
When an ideal $I$ and a filter $F$ are conjugate, we have
$$
1\in I\; \;\Longleftrightarrow\;\;  0\in F
\;\; \Longleftrightarrow\;\;  (I,F)= (L,L).
$$
We shall also write  $L/(I,F)$ for
$L'=  L/(J= 0,U= 1)$ .
By Proposition
  \ref{propIdealFiltre}, a \homo $\varphi$  from
$L$ to another lattice $L_1$ satisfying
$\varphi(J)= \{0\}$ and $\varphi(U)= \{1\}$ can be factorised in an
unique way through the quotient $L'$.

As shown by the example of totally ordered sets
a quotient of distributive lattices is not in general
characterised by the equivalence classes of $0$ and $1$.

\ms Classically a {\em prime ideal} $I$ of a lattice
is an ideal whose complement $F$ is a filter (which is then
a {\em prime filter}). This can be expressed by
$$ 1\notin I\qquad {\rm and}\qquad (x\vi y)\in I\; \Longrightarrow \;
(x\in I{\rm \; or\; }y\in I)\qquad\qquad(*)
$$
which can also be expressed by saying that $I$ is the kernel
of a morphism from $L$ into the lattice with two elements
written $\Deux$.
Constructively, at least in the case where $L$ is discrete,
it seems natural to take the definition $(*)$,
where ``or'' is used constructively. The notion of prime filter
is then defined
  in a dual way.

\begin{definition}
\label{defProiT} Let $L$ be a distributive lattice.
\begin{itemize}
\item An {\em \proiz} in  $L$ is given by a pair $(J,U)$ of
finite subsets of $L$. We consider this as an incomplete specification
for a prime ideal $P$ satisfying $J\subseteq P$ and $U\cap
P= \emptyset$.
\item  To any \proi $(J,U)$  we can associate a pair $(I,F)$
as described in Proposition \ref{propIdealFiltre} where $I$
is an ideal, $F$ is a filter and $I,F$ are conjugate.
\item We say that the \proi $(J,U)$ {\em collapses}
iff we have $I = F = L$.
This means that the quotient lattice $L'= T/(J= 0,U= 1)$ is a
singleton
\cad  $1 \leq_{L'}0$, which means also $U\vda J$.
\end{itemize}
\end{definition}

\begin{theorem}
\label{lemColSimT1} {\em (Simultaneous collapse for \proisz)} Let
$(J,U)$ be an \proi for a lattice $L$ and $x$ be an element of $L$.
If the \prois $(J\cup\{x\},U)$ and
$(J,U\cup\{x\})$  collapse, then so does $(J,U)$.
\end{theorem}
\begin{proof}{Proof}
   We have two finite subsets $J_0,J_1$ of $J$
and two finite subsets $U_0,U_1$ of $U$ such that
$$ x,\; U_0 \vda J_0\quad{\rm  and}\quad   U_1  \vda x,\; J_1
$$
hence
$$ x,\; U_0,\; U_1 \vda J_0,\; J_1\quad{\rm  and}\quad  U_0,\; U_1 \vda
x,\;J_0,\; J_1
$$
By the cut rule
$$  U_0,\; U_1 \vda J_0,\; J_1
$$
\end{proof}

Notice the crucial role of the cut rule.

\subsection{Distributive lattices and \entrels}
\label{subsecTrd2}
An interesting way to analyse the description of distributive lattices
defined by generators and relations is to consider the relation
$A \vda B$ defined on the set $\Pf(L)$ of finite subsets of a
lattice $L$.
Indeed if  $S\subseteq L$ generates the lattice $L$,
  then the relation  $\vda$ on $\Pf(S)$ is enough to characterise
the lattice $L$, because any formula on $S$ can be rewritten,
in normal conjunctive form (inf of sups in $S$) and normal
disjunctive form (sup of infs in $S$). Hence if we want to compare
two elements of the lattice generated by $S$ we write the first
in normal disjunctive form, the second in normal conjunctive form,
and we notice that
$$ \Vu_{i\in I}\left(\Vi A_i \right)\; \leq \; \Vi_{j\in J}\left(\Vu B_j
\right)
\qquad \Longleftrightarrow\qquad  \&_{(i,j)\in I\times J}\;  \left(
A_i \vda  B_j\right)
$$

\begin{definition}
\label{defEntrel}
For an arbitrary set $S$, a relation over  $\Pf(S)$ which is
reflexive, monotone and transitive (see page  \pageref{remotr}) is
called an {\em entailment relation.}
\end{definition}

The notion of \entrels goes back to Gentzen sequent calculus, where
the rule $(T)$ (the cut rule) is first explicitly stated, and
plays a key role. The connection with distributive lattices has been
emphasized in \cite{cc,cp}.
The following result (cf. \cite{cc}) is fundamental. It says that the
three properties of entailment relations are exactly the ones needed
in order to have a faithfull interpretation in distributive lattices.

\begin{theorem}
\label{thEntRel1} {\rm  (fundamental theorem of \entrelsz)} Let $S$ be a
set
with an entailment relation
$\vdash_S$ over $\Pf(S)$. Let $L$ be the lattice defined by generators
and relations as follows: the generators are the elements of $S$ and the
relations are
$$ \bigwedge A\leq \bigvee B
$$
whenever $A\; \vdash_S \; B$. For any finite subsets $A$ and $B$ of $S$
we have
$$  A\; \vdash_L \;  B
\; \Equ \; A\; \vdash_S \;  B.
$$
\end{theorem}
\begin{proof}{Proof}
We give an explicit possible description of the lattice $L$. The
elements
of $L$ are represented by finite sets of finite sets of elements of $S$
$$X= \{A_1,\dots,A_n\}$$
(intuitively $X$ represents $\Vi A_1\vu\cdots\vu\Vi A_n$).
We define then inductively the relation
  $A\prec Y$ with $A\in \Pf(S)$ and $Y\in L$
(intuitively $\Vi A\leq \Vu_{C\in Y} \left(\Vi C\right) $)
\begin{itemize}
\item if $B\in Y$ and $B\subseteq A$ then $A\prec Y$
\item if $A\vdash_S y_1,\dots,y_m$ and $A,y_j\prec Y$ for
$j= 1,\ldots,m$ then $A\prec Y$
\end{itemize}
It is easy to show that if
  $A\prec Y$ and $A\subseteq A'$ then we have
also $A'\prec Y.$ It follows that  $A\prec Z$ holds
whenever $A\prec Y$ and $B\prec Z$
for all $B\in Y$. We can then define $X\leq Y$ by
$A\prec Y$ for all $A\in X$ and one can then check that
$L$ is a distributive lattice\footnote{$L$ is actually the quotient of
$\Pf(\Pf(S))$  by the equivalence relation: $X\leq Y$ and $Y\leq X$.}
for the operations
$$0 =  \emptyset,~~~~1 =  \{\emptyset\},~~~~~X\vee Y =  X\cup Y,~~~~~
X\wedge
Y =  \{ A \cup B~|~A\in X,~B\in Y\}.
$$
For establishing this one first show that if
$C\prec X$ and $C\prec Y$ we have
$C\prec X\vi Y$ by induction on the proofs of
$C\prec X$ and $C\prec Y$.
We notice then that if
  $A\vdash_S y_1,\dots,y_m$ and $A,y_j\vdash_S B$
for all $j$ then $A\vdash_S B$ using $m$ times the cut rule.
It follows that if we have
$A\vdash_L B$, \cad
$A\prec \{\{b\}~|~b\in B\}$, then we have also $A\vdash_S B$.
\end{proof}

As a first application, we give the description of the Boolean
algebra generated by a distributive lattice.
A Boolean algebra can be seen as a distributive lattice with
a complement operation $x\mapsto \overline{x}$ such that
$x\vi \overline{x}= 0$ and $x\vu \overline{x}= 1$.
The application $x\mapsto \overline{x}$ is then a map from the lattice
to its dual.
\begin{proposition}
\label{propTrBoo}
Let $L$ be a distributive lattice. There exists a free
Boolean algebra generated by $L$. It can be described as
the distributive lattice generated by the set
$L_1= L\cup\overline{L}$~\footnote{$\overline{L}$  is a disjoint copy of $L$.}
with the \entrel $\;\vdash_{L_1}\;$
defined as follows:
if $A,B,A',B'$  are finite subsets of $L$ we have
$$A,\overline{B}\;\vdash_{L_1}\; A',\overline{B'}\equidef A,B'\vda
A',B\quad
{\rm in} \;  L
$$
If we write $L_{Bool}$ for this lattice (which is a Boolean algebra), there
is
a natural embedding of $L_1$  in $L_{Bool}$
and the entailment relation of $L_{Bool}$
induces on  $L_1$ the relation  $\,\vdash_{L_1}\,$.
\end{proposition}
\begin{proof}{Proof}
See \cite{cc}.
\end{proof}
Notice that by Theorem \ref{thEntRel1} we have
$x\;\vdash_{L}\; y$ \ssi $x\;\vdash_{L_1}\; y$ hence the canonical
map $L\rightarrow L_1$ is one-to-one and $L$
can be identified to a subset of $L_1$.

\subsection{Spectrum and completeness theorem}

  The {\em spectrum} of the lattice $L$, written  $\Spec(L)$ is defined
as the set $\Hom(L,\Deux)$. It is isomorphic to the ordered set of
all detachable prime ideals. The order relation is then
reverse inclusion.
We have $\Spec(\Deux)\simeq \Un$,  $\Spec(\Trois)\simeq \Deux$,
$\Spec(\Quatre)\simeq \Trois$, etc\ldots

\begin{proposition}
\label{propTr2} \Tcgi
If $(J,U)$ is an \proi which does not collapse
then there exists
$\varphi\in\Spec(L)$
such that $J\subseteq \varphi^{-1}(0)$  and
$U\subseteq \varphi^{-1}(1)$.
In particular if $a\not\leq b$, there exists $\varphi\in\Spec(L)$ such
that
$\varphi(a)= 1$ and $\varphi(b)= 0$.
Also, if $L\neq \Un$, $\Spec(L)$ is nonempty.
\end{proposition}

\begin{proof}{Proof}
This follows from the completeness theorem for geometric theories
(see Appendix).
\end{proof}

A corollary is the following representation theorem
(Birkhoff theorem)
\begin{theorem}
\label{thRep} {\em (Representation theorem)}
\Tcgi The map
$\theta_L: L\rightarrow \cP(\Spec(L))$ defined by
$a\mapsto \left\{\varphi\in \Spec(L)\; ;\; \varphi(a)= 1 \right\}$
is an injective map of distributive lattice. This means that any
distributive
lattice can be represented as a lattice of subsets of a set.
\end{theorem}

Another corollary is the following proposition.
\begin{proposition}
\label{propRep2}
\Tcgi Let $\varphi:L\rightarrow L'$ a map of \trdisz; $\varphi$
is injective \ssi $\Spec(\varphi):\Spec(L')\rightarrow \Spec(L)$ is
surjective.
\end{proposition}
\begin{proof}{Proof}
We have the equivalence
$$ a\not=  b\quad \Longleftrightarrow\quad  a\vi b\not=  a\vu b\quad
\Longleftrightarrow\quad
a\vu b\not\leq a\vi b
$$
Assume that $\Spec(\varphi)$ is surjective. If $a\not=  b$ in $L$, take
$a'= \varphi(a)$, $b'= \varphi(b)$ and let
   $\psi\in\Spec(L)$ be such that $\psi(a\vu b)= 1$ and
  $\psi(a\vi b)= 0$. Since  $\Spec(\varphi)$ is surjective there exists
$\psi'\in\Spec(L')$ such that $\psi= \psi'\varphi$ hence  $\psi'(a'\vu
b')= 1$ and
  $\psi'(a'\vi b')= 0$, hence $a'\vu b'\not\leq a'\vi b'$ and $
a'\not=  b'$.\\
Suppose that  $\varphi$ is injective. We identify $L$ to a sublattice of
$L'$. If $\psi\in\Spec(L)$, take $I= \psi^{-1}(0)$ and $F= \psi^{-
1}(1)$.
By the compactness theorem (see appendix),
there exists $\psi'\in\Spec(L')$ such that $\psi'(I)=0$ and
  $\psi'(F)=  1$, which means $\psi= \psi'\circ\varphi$.
\end{proof}

Of course, these three last results are hard to interpret in
a computational way. An intuitive interpretation is that we
can proceed ``as if'' any distributive lattice is a lattice of
subsets of a set. The goal of Hilbert's program is to give
a precise meaning to this sentence, and explain  what is meant
by ``as if'' there.

\section{Krull dimension of \trdisz}
\label{subsecTrd3}

\subsection{Definition of $\Kr_\ell(L)$}

To develop a suitable constructive theory of the Krull dimension
of a distributive lattice we have to find a constructive counterpart of
the notion of increasing chains of prime ideals.

\begin{definition}
To any distributive lattice $L$ and $\ell\in \mathbb{N}$ we associate a
distributive lattice
$\Kr_{\ell}(L)$ which is the lattice defined by the generators
$\varphi_i(x)$ for $i\leq \ell$ and $x\in L$ (thus we have
$\ell+1$ disjoint copies of $L$ and we let $\varphi_i$ be
the bijection between $L$ and the $i$th copy) and relations
\begin{itemize}
\item $\vdash \varphi_i(1)$
\item $\varphi_i(0)\vdash$
\item $\varphi_i(a),\varphi_i(b)\vdash \varphi_i(a\wedge b)$
\item $\varphi_i(a\vee b)\vdash \varphi_i(a),\varphi_i(b)$
\item $\varphi_i(a)\vdash \varphi_i(b)$ whenever $a\leq b$ in $L$
\item $\varphi_{i+1}(a)\vdash \varphi_{i}(a)$ for $i<\ell$
\end{itemize}
 Let $S$ be the disjoint union $\bigcup \varphi_i(L)$ and
$\vdash _S$ the entailment relation generated by these relations.
\end{definition}

 From this definition, we get directly the following theorem.

\begin{theorem}
\label{univ}
The maps $\varphi_i$ are morphisms from the lattice
$L$ to the lattice $\Kr_{\ell}(L)$. Furthermore
the lattice  $\Kr_{\ell}(L)$ with the maps $\varphi_i$ is then
a solution of the following universal problem:
to find a \trdi $K$ and $\ell+1$ \homos
$\varphi_0\geq \varphi_1\geq \cdots\geq \varphi_\ell$ from $L$ to $K$
such that, for any lattice $ L'$ and any morphism $\psi_0\geq
\psi_1\geq \cdots\geq \psi_\ell\in\Hom(L,L')$ we have one and only one
morphism $\eta :K\rightarrow L'$ such that $\eta \varphi_0= \psi_0$,
$\eta \varphi_1= \psi_1,$ $\ldots$, $\eta \varphi_\ell= \psi_\ell$.
\end{theorem}

 The next theorem is the main result of this paper, and uses
crucially the notion of entailment relation.

\begin{theorem}
\label{thKrJoy}
If $U_i$ and $J_i$ $(i= 0,\ldots,\ell)$ are finite subsets
of $L$ we have in $\Kr_{\ell}(L)$
$$ \varphi_0(U_0)\wedge\ldots\wedge\varphi_\ell(U_\ell) \leq
\varphi_0(J_0)\vee\ldots\vee\varphi_\ell(J_\ell)
$$
\ssi
$$ \varphi_0(U_0),\ldots,\varphi_\ell(U_\ell) \,\vdash_S\,
\varphi_0(J_0),\ldots,\varphi_\ell(J_\ell)
$$
\ssi there exist $x_1,\ldots,x_\ell\in L$ such that
  (where $\vda$ is the \entrel of $L$):
$$\begin{array}{rcl}
x_1,\;  U_0& \vda  &  J_0  \\
x_2,\;  U_1& \vda  &  J_1 ,\; x_1  \\
     \vdots\quad & \vdots  & \quad\vdots  \\
x_\ell,\;  U_{\ell-1}& \vda& J_{\ell-1} ,\; x_{\ell-1}  \\
  U_\ell& \vda  &  J_\ell ,\; x_\ell  \\
\end{array}$$
\end{theorem}

\begin{proof}{Proof}
The equivalence between the first and the second statement
follows from Theorem \ref{thEntRel1}.

We show next that the relation on $\Pf(S)$ described in the
statement of the theorem is indeed an entailment relation. The only
point that needs explanation is the cut rule. To simplify notations,
we take
$\ell= 3.$
We have then 3 possible cases, and we analyse only one case, where
$X,\varphi_1(z)\vdash_S Y$ and $X\vdash_S Y,\varphi_1(z)$, the other
cases being similar.
By hypothesis we have $x_1,x_2,x_3,y_1,y_2,y_3$ such that
$$\begin{array}{rclcrcl}
  x_1,\; U_0& \vda  & J_0  &\qquad &
  y_1,\; U_0& \vda  & J_0  \\
  x_2,\; U_1,\; z& \vda  & J_1 ,\; x_1  &&
  y_2,\; U_1& \vda  & J_1 , y_1 ,\; z \\
  x_3,\; U_{2}& \vda  & J_{2} ,\; x_{2}  &&
  y_3, U_{2}& \vda  & J_{2} , y_{2}  \\
  U_3& \vda  & J_3 ,\; x_3  &&
  U_3& \vda  & J_3 ,\; y_3
\end{array}$$
The two \entrels on the second line give
$$\begin{array}{rclcrcl}
  x_2,\;y_2,\; U_1,\; z& \vda  & J_1 ,\; x_1,\;y_1 &\qquad
  x_2,\;y_2,\; U_1& \vda  & J_1 ,\; x_1 ,\; y_1 ,\; z \end{array}$$
hence by cut
$$\begin{array}{rclcrcl}
  x_2,\;y_2,\; U_1& \vda  & J_1 ,\; x_1,\;y_1
\end{array}$$
\cad
$$\begin{array}{rclcrcl}
  x_2\vi y_2,\; U_1& \vda  & J_1 ,\; x_1\vu y_1
\end{array}$$
Finally, using distributivity
$$\begin{array}{rclcrcl}
  (x_1\vu y_1),\; U_0& \vda  & J_0   \\
  (x_2\vi y_2),\; U_1& \vda&J_1,\;(x_1\vu y_1)\\
  (x_3\vi y_3),\; U_{2}& \vda  & J_{2},\;(x_2\vi y_2)  \\
  U_3& \vda  & J_3, \;(x_3\vi y_3)
\end{array}$$
and hence $\varphi_0(U_0),\dots,\varphi_3(U_3)\,\vdash_S\,
\varphi_0(J_0),\dots,\varphi_3(J_3)$.\\
Finally it is left to notice that the entailment relation
we have defined is clearly the least possible relation
ensuring the  $\varphi_i$ to form a non-increasing chain of
morphisms.
\end{proof}

  Notice that the morphisms $\varphi_i$ are injective: it is easily
seen that for $a,b \in L$ the relation
$\varphi_i(a)\vdash_S\varphi_i(b)$ implies
$a\vda b$, and hence that  $\varphi_i(a)= \varphi_i(b)$ implies
$a= b.$

\subsection{Partially specified chains of prime ideals}
\begin{definition}
\label{defprochTreil}
In a \trdi $L$, a {\em partial specification for a chain of prime ideals}
(that we shall call {\em \procz}) is defined as follows.
An \prolo $\ell$ is a list of $\ell+1$ \prois of $L$:
$\cC= ((J_0,U_0),\ldots,(J_\ell,U_\ell))$.
An \prolo $0$ is nothing but an {\em \proiz}.
\end{definition}

We think of an \prolo $\ell$ as a partial specification of
an increasing chains of prime ideals $P_0,\ldots,P_\ell$
such that $J_i\subseteq P_i$, $U_i\cap P_i= \emptyset$,
$(i= 0,\ldots,\ell)$.

\begin{definition}
We say that an \proc $((J_0,U_0),\ldots,(J_\ell,U_\ell))$
{\em collapses} \ssi
we have in $\Kr_{\ell}(L)$
$$ \varphi _0(U_0),\dots,\varphi_{\ell}(U_{\ell}) \vdash _S
   \varphi _0(J_0),\dots,\varphi_{\ell}(J_{\ell})$$
\end{definition}

Thus an \proc $((J_0,U_0),\ldots,(J_\ell,U_\ell))$ collapses in
$L$ \ssi the \proi
  $\cP= (\varphi_0(J_0),\ldots,\varphi_\ell(J_\ell);
\varphi_0(U_0),\ldots,\varphi_\ell(U_\ell))$
collapses in $\Kr_\ell(L)$. From the completeness theorem we deduce the
following result which justifies this idea of
partial specification.
\begin{theorem}
\label{th.nstformelTreil} {\em (formal \nst for chains of prime ideals)}
\Tcgi Let $L$ be a  \trdi and
$((J_0,U_0),\ldots,(J_\ell,U_\ell))$ be an \proc in $L$. \Propeq
\begin{itemize}
\item [$(a)$] There exist $\ell+1$ prime ideals $P_0\subseteq
\cdots\subseteq
P_\ell$ such that
  $J_i\subseteq P_i$, $U_i\cap P_i= \emptyset $, $(i= 0,\ldots,\ell)$.
\item [$(b)$] The \proc does not collapse.
\end{itemize}
\end{theorem}

\begin{proof}{Proof}
If $(b)$ holds then  the \proi
  $\cP= (\varphi_0(J_0),\ldots,\varphi_\ell(J_\ell);
\varphi_0(U_0),\ldots,\varphi_\ell(U_\ell))$
does not collapse in $\Kr_{\ell}(L)$. It follows then from Proposition
\ref{propTr2} that there exists $\sigma\in \Spec(\Kr_{\ell}(L))$ such that
$\sigma$ is $0$ on $\varphi_0(J_0),\ldots,\varphi_\ell(J_\ell)$
and $1$ on $\varphi_0(U_0),\ldots,\varphi_\ell(U_\ell))$. We
can then take $P_i = (\sigma\circ\varphi_i)^{-1}(0)$.
That $(a)$ implies $(b)$ is direct.
\end{proof}
\subsection{Krull dimension of a \trdi }
\label{subsubsecDimTrdi}

\begin{definition}
\label{defDiTr}~
\begin{itemize}
\item [1)] An {\em \proelz} in a  \trdi $L$ is an \proc of the form
$$ ((0,x_1),(x_1,x_2),\ldots,(x_\ell,1))
$$
(with $x_i$ in $L$).
\item [2)] A \trdi $L$ is {\em of dimension $\leq \ell-1$}
iff it satisfies one of the equivalent conditions
\begin{itemize}
\item Any \proel of length $\ell$ collapses.
\item For any sequence $x_1,\dots,x_\ell\in L$ we have
$$
\varphi_0(x_1),\dots,\varphi_{\ell-1}(x_{\ell})
\vda
\varphi_1(x_1),\dots,\varphi_{\ell}(x_{\ell})
$$
in $\Kr_{\ell}(L)$,
\end{itemize}
\end{itemize}
\end{definition}

 The following result shows that this definition coincides
with the classical definition of Krull dimension for lattices.

\begin{theorem}
The completeness theorem implies that the Krull dimension of a
lattice $L$ is $\leq \ell -1$ if, and only if, there is no
strictly increasing chains of prime ideals of length $\ell$.
\end{theorem}

 Using Theorem \ref{thKrJoy}, we get the following characterisation.

\begin{theorem}
\label{dimLatt}
A distributive lattice $L$ is of Krull dimension $\leq {\ell}-1$
if, and only if, for all $x_1,\dots,x_\ell\in L$
there exist $a_1,\dots,a_\ell\in L$ such that
$$ a_1\wedge x_1 = 0,~~~
   a_2\wedge x_2\leq a_1\vee x_1,~~~,\dots,
   a_{\ell}\wedge x_{\ell}\leq   a_{\ell-1}\vee x_{\ell-1},~~~
   1 = a_{\ell}\vee x_{\ell}$$
\end{theorem}

 In this way we have given a concrete form of the statement
that the \trdi $L$ has a dimension $\leq \ell -1$ in the form
of an  existence of a sequence of inequalities.

  In particular the \trdi $L$ is of dimension $\leq -1$ \ssi $1= 0$
in $L$, and it is of dimension
$\leq 0$ \ssi $L$ is a Boolean algebra (any element has a complement).

\ss We have furthermore.

\begin{lemma}
\label{lemDimGen}
A \trdi $L$ generated by a set $G$ is of dimension $\leq \ell-1$
\ssi for any sequence $x_1,\dots,x_\ell\in G$
$$
\varphi_0(x_1),\dots,\varphi_{\ell-1}(x_{\ell})
\vda
\varphi_1(x_1),\dots,\varphi_{\ell}(x_{\ell})
$$
in $\Kr_{\ell}(L)$.
\end{lemma}
Indeed using distributivity, one can deduce
$$ a\vu a',A\vda b\vu b',B\qquad\qquad
   a\wedge a',A\vda b\wedge b',B$$
from $a,A\vda b,B$ and $a',A\vda b',B$. Furthermore any element of $L$
is an inf of sups of elements of $G$.

\subsection{Implicative lattice}

 A lattice $L$ is said to be an {\em implicative lattice} \cite{Curry} or
{\em Heyting algebra} \cite{Johnstone} if, and only if, there is
a binary operation $\rightarrow$ such that
$$ a\wedge b\leq c \Longleftrightarrow a\leq b\rightarrow c$$

\begin{theorem}
\label{KrImpl}
If $L$ is an implcative lattice,
we have in $\Kr_{\ell}(L)$
$$ \varphi _0(U_0),\dots,\varphi_{\ell}(U_{\ell}) \vdash _S
   \varphi _0(J_0),\dots,\varphi_{\ell}(J_{\ell})$$
if, and only if,
$$
1 = u_{\ell}\rightarrow
   (j_{\ell} \vee (u_{\ell -1} \rightarrow (j_{\ell -1}\vee \dots (u_0\rightarrow j_0))))
$$
where $u_j = \wedge U_j$ and $j_k = \vee J_k.$
\end{theorem}

 In the case where $L$ is an implicative lattice, we can write
explicitely that $L$ is of dimension $\leq {\ell}-1$ as an
identity. For instance that $L$ is of dimension $\leq 0$
is equivalent to the identity
$$ 1 = x \vee \neg x$$
where $\neg x = x\rightarrow 0$ and that $L$ is of dimension
$\leq 1$ is equivalent to the identity
$$ 1 = x_2 \vee (x_2 \rightarrow (x_1\vee \neg x_1))$$
and so on.

\begin{corollary}
\label{dimImpl}
An implicative lattice $L$ is of dimension $\leq {\ell}-1$ if, and only if,
for any sequence $x_1,\dots,x_{\ell}$
$$ 1 = x_{\ell}\vee (x_{\ell}\rightarrow \dots (x_2 \vee (x_2 \rightarrow (x_1\vee \neg x_1)))\dots)
$$
\end{corollary}

\subsection{Decidability}
\label{decidability}

 To any distributive lattice $L$ we have associated a family
of distributive lattices $\Kr_{\ell}(L)$ with a complete description
of their ordering. A lattice is {\em discrete} if, and only if,
its ordering is decidable, which
means intuitively that there is an algorithm to decide the ordering
(or, equivalently, the equality) in this lattice. It should be intuitively clear that
we could find a discrete lattice $L$ such that $\Kr_1(L)$ is not
discrete since, by \ref{thKrJoy},
the ordering on $\Kr_1(L)$ involves an existential quantification
on the set $L$, that  may be infinite (this point is discussed in
\cite{BoileauJoyal}, with another argument).
However we can use the characterisation of
Theorem \ref{thKrJoy} to give a general sufficient condition ensuring
that all $\Kr_{\ell}(L)$ are discrete.

\begin{theorem}
\label{decKr}
Suppose that the lattice $L$ is a discrete implicative lattice
then each $\Kr_{\ell}(L)$ is discrete.
\end{theorem}

\begin{proof}{Proof}
This is direct from Theorem \ref{KrImpl}.
\end{proof}

\subsection{Dimension of Spectral Spaces}

This subsection is written from a classical point of view.
Following \cite{Hochster}, a topological space $X$ is called a spectral space if it satisfies the
following conditions:
(a) $X$ is a compact $T\sb 0$-space; (b) $X$ has a compact open basis which is closed under
finite intersections;
(c) each irreducible closed subspace of $X$ has a generic point. $\Spec(R)$, with the Zariski
topology, is spectral for any commutative ring $R$ with identity. Similary, if we take
for basic open the sets $U_a = \{\phi\in\Spec(L)~|~\phi(a) = 1\}$
then $\Spec(L)$ is spectral for any distributive lattice. The compact
open subsets of a spectral space form a distributive lattice, and
it is well-known \cite{Stone,Johnstone}
that, if $L$ is an arbitrary distributive lattice, then $L$ is isomorphic to the lattice of
compact open subsets of the space $\Spec(L)$.

If $U,V$ are open subsets of a topological space $X$ we define
$U\rightarrow V$ to be the largest open $W$ such that
$W\cap U\subseteq V$
and $\neg~U = U\rightarrow \emptyset$. In a classical
setting a spectral space
$X$ is said to be of dimension $\leq {\ell}-1$ if, and only if, there is no strictly
increasing chains of length $\ell$ of irreducible closed subsets of $X$.
We can reformulate Theorem \ref{dimLatt} as follows.

\begin{theorem}
\label{dimSpec}
A spectral space $X$ is of dimension $\leq {\ell}-1$ if, and only if,
for any compact open subsets $x_1,\dots,x_{\ell}$ of $X$
$$ X = x_{\ell}\vee (x_{\ell}\rightarrow \dots (x_2 \vee (x_2 \rightarrow (x_1\vee \neg x_1)))\dots)
$$
\end{theorem}

\subsection{Connections with Joyal's definition}
\label{subsubsecJoyal2}

  Let $L$ be a distributive lattice, Joyal \cite{esp} gives the following
definition of $\dim(L)\leq \ell$. Let
$\varphi^\ell_i:L\rightarrow \Kr_\ell(L)$ be the $\ell+1$ universal
morphisms.
By universality of $\Kr_{\ell+1}(L)$, we have
$\ell+1$ morphisms $\sigma_i:\Kr_{\ell+1}(L)\rightarrow
\Kr_\ell(L)$
such that $\sigma_i\circ \varphi^{\ell+1}_j =  \varphi^\ell_j$ if
$j\leq i$
and $\sigma_i\circ \varphi^{\ell+1}_j =  \varphi^\ell_{j-1}$ if $j>i$.
Joyal defines then $\dim(L)\leq \ell$ to mean that
$(\sigma_0,\dots,\sigma_\ell):\Kr_{\ell+1}(L)\rightarrow
\Kr_\ell(L)^{\ell+1}$
is injective. This definition can be motivated by Proposition
\ref{propRep2}: the elements in the image of
$Sp(\sigma_i)$ are the chains of prime ideals
$(\alpha_0,\dots,\alpha_{\ell})$
with $\alpha_i= \alpha_{i+1}$, and $Sp(\sigma_0,\dots,\sigma_\ell)$
is surjective \ssi for any chain $(\alpha_0,\dots,\alpha_{\ell})$
there exists $i<\ell$ such that $\alpha_i= \alpha_{i+1}$. This means
exactly
that there is no nontrivial chain of prime ideals of length $\ell+1$.
Using the completeness theorem, one can then see the equivalence
with Definition \ref{defDiTr}.
One could check directly this equivalence using a constructive
metalanguage, but
for lack of space, we shall not present here this argument.
Similarly, it would be possible to establish the equivalence of our
definition with the one of Espa\~nol \cite{esp} (here also, this
connection is clear via the completeness theorem).

\section{Zariski and Krull lattices of a commutative ring}
\label{secZariKrull}
\subsection{Zariski lattice}
\label{subsubsecZar}
Let $R$ be a commutative ring. We write
$\gen{J}$  or explicitly
$\gen{J}_R$ for the ideal of $R$ generated by the subset $J\subseteq R$.
We write
$\cM(U)$ for the monoid\footnote{A monoid will always be
multiplicative.}
generated by the subset $U\subseteq R$.
Given a commutative ring $R$ the {\em Zariski lattice}  $\Zar(R)$ has
for elements the radicals of finitely generated
ideals (the order relation being inclusion).
It is well defined as a lattice.
Indeed $\sqrt{I_1}= \sqrt{J_1}$ and $\sqrt{I_2}= \sqrt{J_2}$
imply
$\sqrt{I_1I_2}=  \sqrt{J_1J_2} $
(which defines $\sqrt{I_1}\vi\sqrt{I_2}$) and
$\sqrt{I_1+I_2}= \sqrt{J_1+J_2}$
(which defines $\sqrt{I_1}\vu\sqrt{I_2}$). The Zariski lattice of $R$ is
always distributive, but may not be discrete, even if $R$ is discrete.
Nevertheless an inclusion  $\sqrt{I_1}\subseteq \sqrt{I_2}$ can always
be certified in a finite way if the ring $R$ is discrete.
This lattice contains all the informations necessary for a constructive
development of the abstract theory of the Zariski spectrum.\\
We shall write $\wi{a}$ for $\sqrt{\gen{a}}$. Given a subset $S$ of $R$
we write $\wi{S}$ for the subset of $\Zar(R)$ the  elements of which are
$\wi{s}$ for
$s\in S$. We have
$\wi{a_1}\vu\cdots\vu\wi{a_m}= \sqrt{\gen{a_1,\ldots,a_m}}$ and
$\wi{a_1}\vi\cdots\vi\wi{a_m}= \wi{a_1\cdots a_m}$. \\
Let $U$ and  $J$ be two finite subsets of $R$, we have
$$ \wi{U}\,\vdash_{\Zar(R)} \wi{J}
\quad\Longleftrightarrow \quad
\prod_{u\in U} u  \in \sqrt{\gen{J}}
\quad\Longleftrightarrow \quad
\cM(U)\cap \gen{J}\neq \emptyset
$$
This describes completely the lattice $\Zar(R)$. More precisely we have:
\begin{proposition}
\label{propZar} The lattice $\Zar(R)$ of a commutative ring $R$ is
(up to isomorphism) the lattice generated by  $(R,\vda)$  where $\vda$
is the least entailment relation over $R$ such that
$$\begin{array}{rclcrclcrcl}
   0   & \vda  &      &\qquad  & x,\; y & \vda  &  x y   \\
         & \vda  &  1 &\qquad &  xy    & \vda  &   x   &\qquad &
   x+y   & \vda  & x,\; y  \\
\end{array}$$
\end{proposition}
\begin{proof}{Proof}
It is clear that the relation $U\vda J$ defined by ``$\cM(U)$
meets $\gen{J}$'' satisfies these axioms. It is also clear that
the entailment relation generated by these axioms contains this
relation. Let us show that this relation is an  \entrelz.
Only the cut rule is not obvious.
Assume that $\cM(U,a)$ meets
$\gen{J}$ and that $\cM(U)$ meets $\gen{J,a}$. There exist then
$m_1,m_2\in \cM(U)$ and $k\in \mathbb{N},x\in R$ such that
$a^k m_1 \in \gen{J},~m_2+ax\in \gen{J}$.
Eliminating $a$ this implies that $\cM(U)$ intersects $\gen{J}.$
\end{proof}
We have $\wi{a}= \wi{b}$ \ssi $a$ divides a power of $b$ and
$b$ divides a power of $a$.
\begin{proposition}
\label{propZar2}
In a commutative ring $R$ to give an ideal of the lattice
$\Zar(R)$ is the same as to give a radical ideal of $R$.
If $I$ is a radical ideal of $R$ one associates the ideal
$${\cI} =  \{ J \in \Zar(R)~|~ J \subseteq I\}$$
of $\Zar(R)$. Conversely if $\cal{I}$ is an ideal of
$\Zar(R)$ one can associate the ideal
$$I= \bigcup _{J\in\cal{I}}J= \{ x\in R~|~\wi{x}\in \cal{I}\},$$
which is a radical ideal of $R.$
In this bijection the prime ideals of the ring correspond to
the prime ideals of the Zariski lattice.
\end{proposition}
\begin{proof}{Proof}
We only prove the last assertion.
If $I$ is a prime ideal of $R$, if $J,J'\in \Zar(R)$ and
$J\vi J'\in \cal{I}$, let $a_1,\dots,a_n\in R$ be some ``generators'' of
$J$ (\cad $J= \sqrt{\gen{a_1,\dots,a_n}}$) and let $b_1,\dots,b_m\in R$
be some
generators of $J'.$ We have
$a_ib_j\in I$ and hence $a_i\in I$ or $b_j\in I$ for all $i,j.$ It
follows
from this (constructively) that we have $a_i\in I$ for all $i$ or
$b_j\in I$ for all $j$. Hence $J\in \cal{I}$ or $J'\in \cal{I}$ and
$\cal{I}$ is a prime ideal of $\Zar(R).$
\\
Conversely if $\cal{I}$ is a prime ideal of $\Zar(R)$
and if we have $\wi{xy}\in \cal{I}$ then
$\wi{x}\vi \wi{y}\in \cal{I}$ and hence $\wi{x}\in \cal{I}$
or $\wi{y}\in \cal{I}$. This shows that $\{ x\in R~|~\wi{x}\in
\cal{I}\}$
is a prime ideal of $R$.
\end{proof}
\subsection{Krull lattices of a commutative ring}

\begin{definition}
\label{defKruA}
We define $\Kru_\ell(R):= \Kr_\ell(\Zar(R))$. This is called the
{\em Krull lattice of order $\ell$} of the ring $R$. We say also
that $R$ is of Krull dimension $\leq {\ell}$  iff the distributive
lattice $\Zar(R)$ is of dimension  $\leq {\ell}$.
\end{definition}
\begin{theorem}
\label{dimRing}
The ring $R$ is of dimension $\leq {\ell}-1$ \ssi
for any $x_1,\dots,x_n\in R$ we have  in
$\Kru _\ell (R)$
$$ \varphi_0(\wi{x_1}),\dots,\varphi_{\ell-1}(\wi{x_{\ell}})
\vda
\varphi_1(\wi{x_1}),\dots,\varphi_{\ell}(\wi{x_{\ell}})
$$
\end{theorem}

\begin{proof}{Proof}
This is a direct consequence of Lemma \ref{lemDimGen} and the fact that
the elements $\wi{x}$ generates $\Zar(R)$.
\end{proof}

\begin{theorem}
\label{corCollaps} Let $\cC= ((J_0,U_0),\ldots,(J_\ell,U_\ell))$ be
a list of $\ell+1$ pairs of finite subsets of $R$, \propeq
\begin{enumerate}
\item  there exist $j_i\in \gen{J_i}$, $u_i\in\cM(U_i)$,
$(i= 0,\ldots,\ell)$,
such that
$$u_0\cdot(u_1\cdot(\cdots(u_\ell+j_\ell)+\cdots)+j_1)+j_0= 0
$$
\item  there exist $L_1,\ldots,L_\ell\in \Zar(R)$ such that
in $\Zar(R)$:
$$\begin{array}{rcl}
  L_1,\; \widetilde{U_0}& \vda  &  \widetilde{J_0}
\\
  L_2,\; \widetilde{U_1}& \vda  &  \widetilde{J_1} ,\;  L_1
\\
\vdots\qquad & \vdots  & \qquad\vdots
\\
L_\ell,\; \widetilde{U_{\ell-1}}&\vda& \widetilde{J_{\ell-1}},\;L_{\ell-
1}
\\
\widetilde{U_\ell}& \vda & \widetilde{J_\ell} ,\; L_\ell
\end{array}$$
\item  there exist $x_1,\ldots,x_\ell\in R$ such that
(for the entailment relation described in Proposition \ref{propZar}):
$$\begin{array}{rcl}
  x_1,\; {U_0}& \vda  &  {J_0}
\\
  x_2,\; {U_1}& \vda  &  {J_1} ,\;  x_1
\\
\vdots\qquad & \vdots  & \qquad\vdots
\\
x_\ell,\; {U_{\ell-1}}&\vda& {J_{\ell-1}},\;x_{\ell-
1}
\\
{U_\ell}& \vda & {J_\ell} ,\; x_\ell
\end{array}$$
\end{enumerate}
\end{theorem}
\begin{proof}{Proof}
It is clear that $1$ entails $3$: simply take
$$x_\ell =  u_\ell+j_\ell,~x_{\ell-1} =  x_\ell u_{\ell-1} +
j_{\ell-1},\dots,~x_0 =  x_1u_0 + j_0
$$
and that $3$ entails $2$.

Let us prove that $2$ implies $3$. We assume:
$$\begin{array}{rcl}
  L_1,\; \widetilde{U_0}& \vda  &  I_0  \\
  L_2,\; \widetilde{U_1}& \vda  &  I_1 ,\;  L_1  \\
  \widetilde{U_2}& \vda  &  I_2 ,\; L_2
\end{array}$$
The last line means that $\cM(U_2)$ intersects $I_2+ L_2$
and hence  $I_2+ \gen{x_2}$ for some element $x_2$ of
$L_2$. Hence we have
$ \wi{U_2} \vda  I_2 ,\; \wi{x_2}$.
Since $\wi{x_2}\leq L_2$ in $\Zar(R)$ we have
  $ \;\wi{x_2},\; \wi{U_1} \vda  I_1 ,\; L_1$.
We have then replaced  $L_2$ by $\wi{x_2}$.
Reasoning as previously one sees that one can replace as
well $L_1$ by a suitable $\wi{x_1}$. One gets then $3$.

 Finally, let us show that $3$ entails $1$: if we have
for instance
$$\begin{array}{rcl}
  x_1,\; {U_0}& \vda  &  I_0  \\
  x_2,\; {U_1}& \vda  &  I_1 ,\;  x_1  \\
  {U_2}& \vda  &  I_2 ,\; x_2  \\
\end{array}$$
by the last line we know that we can find $y_2$ both in
the monoid $M_2 = \cM(U_2) + \gen{I_2}$ and in $\gen{x_2}.$
Since $y_2\vdash x_1$
$$  y_2,\; {U_1} \vda    I_1 ,\;  x_1$$
and since $y_2\in M_2$ we can find $y_1$ both in the monoid
$M_1 = M_2\cM(U_1) + \gen{I_1}$ and in $\gen{x_1}$. We have
$y_1\vdash x_1$ and hence
$$ y_1,U_0\vdash I_0$$
and since $y_1\in M_1$ this implies
$0\in M_1\cM(U_0)+\gen{I_0}$ as desired.
\end{proof}

\begin{corollary}
\label{KrulldimRing}
A ring $R$ is of Krull dimension $\leq {\ell}-1$ iff
for any sequence $x_1,\dots,x_{\ell}$ there exist
 $a_1,\ldots,a_\ell\in R$
and  $m_1,\ldots,m_\ell\in \N$ such that
$$ x_1^{m_1}(\cdots(x_\ell^{m_\ell}(1+a_\ell x_\ell)+\cdots)+a_1x_1)=
0
$$
\end{corollary}

\begin{proof}{Proof}
By Theorem \ref{dimRing}, we have in $\Kru _\ell (R)$
$$ \varphi_0(\wi{x_1}),\dots,\varphi_{\ell-1}(\wi{x_{\ell}})
\vda
\varphi_1(\wi{x_1}),\dots,\varphi_{\ell}(\wi{x_{\ell}})
$$
we can then apply Theorem \ref{corCollaps} to the
elementary idealistic chain
$$((0,\wi{x_1}),(\wi{x_1},\wi{x_2}),\dots,(\wi{x_{\ell}},1))$$
and we get in this way $j_i\in \gen{x_i},j_0 = 0$ and
$u_i\in \cM(x_{i+1}),u_\ell = 1$ such that
$$u_0\cdot(u_1\cdot(\cdots(u_\ell+j_\ell)+\cdots)+j_1)+j_0= 0$$
as desired.
\end{proof}

This concrete characterisation of the Krull dimension of a ring
can be found in \cite{lom}, where it is derived using dynamical
methods \cite{clr}.

\begin{lemma}
\label{ZarImpl}
If $R$ is coherent and noetherian then  $\Zar(R)$ is an implicative
lattice.
\end{lemma}

\begin{proof}{Proof}
Let $L\in \Zar(R)$, radical of an ideal generated by
elements $y_1,\dots,y_n$  and $x\in R$, we show how to define an element
$\wi{x}\rightarrow L\in \Zar(R)$ such that, for any $M\in \Zar(R)$
$$
M\wedge \wi{x}\leq L \Equ M\leq \wi{x}\rightarrow L
$$
For this, we consider the sequence of ideals
$$
I_k = \{ z\in R~|~z x^k \in \gen{y_1,\dots,y_n}\}
$$
Since $R$ is coherent, each $I_k$ is finitely generated. Since
furthermore $R$ is noetherian and $I_k\subseteq I_{k+1}$ the sequence
$I_k$ is stationary and $\bigcup _k I_k$ is finitely generated.
We take for $\wi{x}\rightarrow L$ the radical of this ideal.

 If $M\in \Zar(R)$ then $M$ is the radical of an ideal
generated by finitely many elements $x_1,\dots,x_m$ and
we can take
$M\rightarrow L =
  (\wi{x_1}\rightarrow L)\wedge\dots\wedge(\wi{x_m}\rightarrow L).$
\end{proof}

\begin{corollary}
If $R$ is coherent, noetherian and
strongly discrete then each lattice $\Kr_n(R)$ is discrete.
\end{corollary}

\begin{proof}{Proof}
Using Theorem \ref{decKr} and Lemma \ref{ZarImpl} we are left to
show that $\Zar(R)$ is discrete. We have $M\leq L$ if, and
only if, $1 = M\rightarrow L$. But to test if an element
of $\Zar(R)$ is equal to the ideal $\gen{1}$ is decidable since
$R$ is strongly discrete.
\end{proof}

The hypotheses of this corollary are satisfied if $R$ is a polynomial ring
$K[X_1,\dots,X_n]$ over a discrete field $K$ \cite{MRR}.

\subsection{Krull dimension of a polynomial ring over a discrete
field}
Let $R$ be a commutative ring, let us say that a sequence
$x_1,\dots,x_\ell$ is {\em singular} if, and only if, there exists
 $a_1,\ldots,a_\ell\in R$
and  $m_1,\ldots,m_\ell\in \N$ such that
$$ x_1^{m_1}(\cdots(x_\ell^{m_\ell}(1+a_\ell x_\ell)+\cdots)+a_1x_1)=
0
$$
A sequence
is {\em pseudo regular} if, and only if, it is not singular.
Corollary \ref{KrulldimRing} can be reformulated as: a ring $R$
is of Krull dimension $\leq {\ell}-1$ if, and only if, any sequence
in $R$ of length $\ell$ is singular.

\begin{proposition}
\label{propKrDimetDegTr}
Let $K$ be a discrete field, $R$ a commutative $K$-algebra, and $x_1$,
\ldots, $x_\ell$ in $R$ algebraically dependent over $K$.
The sequence $x_1,\dots,x_\ell$ is singular.
\end{proposition}
\begin{proof}{Proof}
Let $Q(x_1,\ldots,x_\ell)= 0$ be a algebraic dependence relation
over $K$. Let us order the nonzero monomials of $Q$ along
the lexicographic ordering. We can suppose that the coefficient of
the first monomial is
$1$. Let
$x_1^{m_1}x_2^{m_2}\cdots x_\ell^{m_\ell}$ be this momial,
it is clear that
$Q$ can be written on the form
$$ Q= x_1^{m_1}\cdots x_\ell^{m_\ell}+
x_1^{m_1}\cdots x_\ell^{1+m_\ell}R_\ell+
x_1^{m_1}\cdots x_{\ell-1}^{1+m_{\ell-1}}R_{\ell-1}+\cdots+
x_1^{m_1}x_2^{1+m_2}R_2+ x_1^{1+m_1}R_1
$$
and this is the desired collapsus.
\end{proof}

Let us say that a ring is of dimension $\ell$ if it is of dimension $\leq {\ell}$
but not of dimension $\leq {\ell}-1$. It follows that we have:
\begin{theorem}
\label{thKDP} Let $K$ be a discrete field. The Krull dimension of the
ring
$K[X_1,\ldots,X_\ell]$ is equal to $\ell$.
\end{theorem}
\begin{proof}{Proof}
Given Proposition \ref{propKrDimetDegTr} it is enough to check that the
sequence $(X_1,\ldots,X_\ell)$ is pseudo regular, which is direct.
\end{proof}

Notice that we got this basic result quite directly from the
characterisation of Corollary \ref{KrulldimRing}, and that our
argument is of course also valid classically (with the usual
definition of Krull dimension).
This contradicts the current opinion that constructive arguments are
necessarily more involved than classical proofs.


\newpage
\markboth{Annex: The \tcg and LLPO}{Annex: The \tcg and LLPO}

\setcounter{section}{1}\setcounter{subsection}{0}
\setcounter{theorem}{0}
\def\thesection{\Alph{section}}
\addcontentsline{toc}{section}{Annex: The \tcg and LLPO}
\section*{Annex: Completeness, compactness theorem, LLPO and geometric theories}

\subsection{Theories and models}

  We fix a set $V$ of {\em atomic propositions} or
{\em propositional letters}. A proposition $\phi,\psi,\dots$
is a syntactical
object built from the atoms $p,q,r\in V$ with the
usual logical connectives
$$ 0,\;\;1,\;\;\phi\wedge \psi,\;\;\phi\vee\psi,\;\;\phi\rightarrow\psi,\;\;
\neg\phi$$
We let $P_V$ be the set of all propositions. Let $F_2$
be the Boolean algebra with two elements.
A {\em valuation}
is a function $v\in F_2^V$ that assigns a truth value
to any of the atomic propositions. Such a valuation can
be extended to a map
$P_V\rightarrow \{0,1\},\;\phi\longmapsto v(\phi)$
in the expected way. A {\em theory} $T$ is a subset of $P_V$.
A {\em model} of $T$ is a valuation $v$ such that $v(\phi)=1$
for all $\phi\in T$.

 More generally given a Boolean algebra $B$ we can define
$B$-valuation to be a function $v\in B^V$. This can be extended
as well to a map $P_V\rightarrow B,\;\phi\longmapsto v(\phi)$.
A {\em $B$-model} of $T$ is a valuation $v$ such that $v(\phi)=1$
for all $\phi\in T$. The usual notion of model is a direct special
case, taking for $B$ the Boolean algebra $F_2$.
For any theory there exists always a
free Boolean algebra over which $T$ is a model, the {\em Lindenbaum}
algebra of $T$, which can be also be defined as the Boolean
algebra generated by $T$, thinking of the elements of $V$ as
generators and the elements of $T$ as relations. The theory $T$
is {\em formally consistent} if, and only if, its Lindenbaum algebra is not trivial.

\subsection{Completeness theorem}

\begin{theorem}
(Completeness theorem) Let $T$ be a theory. If $T$ is formally
consistent then $T$ has a model.
\end{theorem}

 This theorem is the completeness theorem for propositional logic.
Such a theorem is strongly related to Hilbert's program, which
can be seen as an attempt to replace the question of existence of
model of a theory by the formal fact that this theory
is not contradictory.

Let $B$ the Lindenbaum algebra of $T$. To prove completeness, it is
enough to find a morphism $B\rightarrow F_2$ assuming that $B$ is not
trivial, wich is the same as finding a prime ideal (which is then
automatically maximal) in $B$. Thus the completeness theorem is a
consequence of the existence of prime ideal in nontrivial Boolean algebra.  Notice
that this existence is clear in the case where $B$ is finite, hence
that the completeness theorem is direct for finite theories.

\subsection{Compactness theorem}

 The completeness theorem for an arbitrary theory
can be seen as a corollary of the following
fundamental result.

\begin{theorem}
(Compactness theorem) Let $T$ be a theory.
If all finite subsets of $T$ have a model then so does $T$.
\end{theorem}

 Suppose indeed that the compactness theorem holds, and let
$T$ be a formally consistent theory. Then an arbitrary finite
subset $T_0$ of $T$ is also formally consistent. Furthermore, we
have seen that this implies the existence of a model for $T_0$.
 It follows then from the compactness theorem
that $T$ itself has a model.

 Conversely, it is clear that the compactness theorem follows
from the completeness theorem, since a theory is formally
consistent as soon as all its finite subsets are.

 A simple general proof of the compactness theorem
is to consider the product topology
on $\{0,1\}^V$ and to notice that the set of models
of a given subset of $T$ is a closed subset. The theorem is then
a corollary of the compactness of the space $W:=\{0,1\}^V$ when
compactness is expressed (in classical mathematics) as:
if a family of closed subsets of $W$ has non-void finite
intersections, then its intersection is non-void.

\subsection{LPO and LLPO}

If $V$ is countable (\cad discrete and enumerable)
we have the following
alternative argument. One writes $V = \{p_0,p_1,\dots \}$
and builds by induction a partial valuation $v_n$
on $\{p_i~|~i<n\}$ such that any finite subset of $T$ has
a model which extends $v_n$, and $v_{n+1}$ extends
$v_n$. To define $v_{n+1}$ one first tries $v_{n+1}(p_n) = 0$.
If this does not work, there is a finite subset of $T$
such that any of its model $v$ that extends $v_n$ satisfies
$v(p_n) = 1$ and one can take $v_{n+1}(p_n) = 1.$

The non-effective part of this argument is contained in
the choice of $v_{n+1}(p_n)$, which demands to give a
gobal answer to an infinite set of (elementary) questions.

Now let us assume also that we can enumerate the infinite
set $T$. We can then build a sequence of finite
subsets of $T$ in a nondecreasing way
$K_0\subseteq K_1\subseteq \dots$ such that any finite
subset of $T$ is a subset of some $K_n$.
Assuming we have construct $v_{n}$ such that all $K_j$'s
have a model extending $v_{n}$, in order to define
$v_{n+1}(p_n)$ we have to give a global answer to the
questions: do all $K_j$'s have a model extending $v_{n+1}$
when we choose $v_{n+1}(p_n)=1$~?
For each $j$ this is an elementary question, having
a clear answer.
More precisely let us define $g_{n}:\N\rightarrow \{0,1\}$
in the following way: $g_{n}(j)=1$ if there is a model
$v_{n,j}$ of $K_{j}$ extending  $v_{n}$ with $v_{n,j}(p_{n})=1$,
else $g_{n}(j)=0$. By induction hypothesis if $g_{n}(j)=0$
then all  $K_{\ell}$ have a model $v_{n,\ell}$ extending
$v_{n}$ with $v_{n,\ell}(p_{n})=1$, and all models  $v_{n,\ell}$
of  $K_{\ell}$ extending  $v_{n}$ satisfy  $v_{n,\ell}(p_{n})=1$
if $\ell\geq j$.
So we can ``construct" inductively
the infinite sequence of partial models $v_{n}$ by using
at each step the non-constructive Bishop's principle LPO
(Least Principle of Omniscience):
given a function $f:\N\rightarrow \{0,1\}$,
either $f=1$ or $\exists j\in\N\;f(j)\neq 1$.
This principle is applied at step $n$ to the function $g_n$.

In fact  we can slightly modify the argument and
use only a combination of Dependant Choice and
of Bishop's principle LLPO (Lesser Limited
Principle of Omniscience), which is known to be
strictly weaker than LPO: given two non-increasing functions
$g,h:\N\rightarrow \{0,1\}$ such that, for all $j$
$$g(j) = 1 \vee h(j) = 1$$
then we have $g=1$ or $h=1$.
Indeed let us define  $h_{n}:\N\rightarrow \{0,1\}$
in a symmetric way: $h_{n}(j)=1$ if there is a model
$v_{n,j}$ of $K_{j}$ extending  $v_{n}$ with $v_{n,j}(p_{n})=0$,
else $h_{n}(j)=0$. Cleraly $g_n$ and $h_n$ are non-increasing functions.
By induction hypothesis, we have for all $j$
$g_n(j) = 1 \vee h_n(j) = 1$. So,
applying LLPO, we can define $v_{n+1}(p_n)=1$
if $g_n=1$ and $v_{n+1}(p_n)=1$ if $h_n=1$.
Nevertheless, we have to use dependant choice in
order to make this choice inifnitely often since the answer
``$g=1$ or $h=1$" given by the oracle LLPO may be
ambiguous.

In a reverse way
it is easy to see that the \tcg restricted to the
countable case implies LLPO.

\subsection{Geometric formulae and theories}

{\em What would have happened if topologies without points
had been discovered before topologies with points, or if
Grothendieck had known the theory of distributive lattices?}  (G. C. Rota \cite{Rota}).


\bigskip

 A formula is {\em geometric} if, and only if, it is built only
with the connectives $0,1,\phi\wedge\psi,\phi\vee\psi$ from
the propositional letters in $V$. A theory if
a (propositional) {\em geometric} theory iff all the formula in $T$
are of the form $\phi\rightarrow \psi$ where $\phi$
and $\psi$ are geometric formulae.

 It is clear that the formulae of a geometric theory $T$ can be seen as
relations for generating a distributive lattice $L_T$ and that
the Lindenbaum algebra of $T$ is nothing else but the free Boolean
algebra generated by the lattice $L_T$. It follows from
Proposition \ref{propTrBoo} that $T$ is formally consistent if, and
only if, $L_T$ is nontrivial. Also, a model of $T$ is nothing else
but an element of $\Spec(L_T)$.

\begin{theorem}
(Completeness theorem for geometric theories) Let $T$ be a geometric
theory. If $T$ generates a nontrivial distributive lattice, then
$T$ has a model.
\end{theorem}

 The general notion of geometric formula allows also existential
quantification, but we restrict ourselves here to the propositional case.
Even in this restricted form, the notion of geometric theory is
fundamental. For instance, if $R$ is a commutative ring, we can
consider the theory with atomic propositions
$D(x)$ for each $x\in R$ and with axioms
\begin{itemize}
\item $D(0_R)\rightarrow 0$
\item $1\rightarrow D(1_R)$
\item $D(x)\wedge D(y)\rightarrow D(xy)$
\item $D(xy)\rightarrow D(x)$
\item $D(x+y)\rightarrow D(x)\vee D(y)$
\end{itemize}
This is a geometric  theory $T$. The model of this theory are
clearly the complement of the prime ideals. What is remarkable is that, while the
existence of models of this theory is a nontrivial fact
which may be dependent on set theoretic axioms (such as dependent axiom of choices)
its formal consistency is completely elementary (as explained in the
beginning of the section \ref{secZariKrull}).
This geometric theory, or the distributive lattice it generates,
can be seen as a point-free description of the Zariski
spectrum of the ring. The distributive lattice
generated by this theory (called in this paper the Zariski lattice of $R$)
is isomorphic to the lattice
of compact open of the Zariski spectrum of $R$, while the Boolean
algebra generated by this theory is isomorphic to the algebra
of the constructible sets.


\end{document}